\documentclass[a4paper]{amsart}

\author[F.~Jahnke]{Franziska Jahnke}
\address{Institute for Mathematical Logic and Foundations, Department of Mathematics and Computer Science,
University of M\"unster, 
Einsteinstraße 62,
48149 M\"unster, Germany}\email{franziska.jahnke@uni-muenster.de}

\author[M.~Ketelsen]{Margarete Ketelsen} 
\address{Max Planck Institute for Mathematics, Vivatsgasse 7, 53111 Bonn, Germany}
\email{ketelsen@mpim-bonn.mpg.de}

\author[F.~Vermeulen]{Floris Vermeulen}
\address{Institute for Mathematical Logic and Foundations, Department of Mathematics and Computer Science,
University of Münster,
Einsteinstraße 62,
48149 M\"unster, Germany}
\email{florisvermeulen.math@gmail.com}

\title[AKE principles in roughly deeply ramified henselian fields]{AKE principles in roughly deeply ramified henselian valued fields}

\usepackage[utf8]{inputenc}
\usepackage{thmtools}
\usepackage[pdfusetitle]{hyperref}

\usepackage[capitalise]{cleveref}

\usepackage{amsmath}
\usepackage{yhmath}
\usepackage{amsthm,stackengine,scalerel}
\usepackage{amssymb}
\usepackage{mathrsfs}
\usepackage{mathtools}
\usepackage{tikz-cd}
\usepackage{tikz}
\usetikzlibrary{arrows,calc,decorations.pathmorphing,decorations.text,patterns,patterns.meta,shapes,fit}
\usepackage{adjustbox}
\usepackage{enumitem}
\usepackage[capitalise]{cleveref}
\usepackage{csquotes}

\usepackage[colorinlistoftodos]{todonotes}

\usepackage{faktor}

\tikzset{curve/.style={settings={#1},to path={(\tikztostart)
			.. controls ($(\tikztostart)!\pv{pos}!(\tikztotarget)!\pv{height}!270:(\tikztotarget)$)
			and ($(\tikztostart)!1-\pv{pos}!(\tikztotarget)!\pv{height}!270:(\tikztotarget)$)
			.. (\tikztotarget)\tikztonodes}},
	settings/.code={\tikzset{quiver/.cd,#1}
		\def\pv##1{\pgfkeysvalueof{/tikz/quiver/##1}}},
	quiver/.cd,pos/.initial=0.35,height/.initial=0}

\newtheoremstyle{mystyle}% name of the style to be used
{}% measure of space to leave above the theorem. E.g.: 3pt
{}% measure of space to leave below the theorem. E.g.: 3pt
{\normalfont}% name of font to use in the body of the theorem \bfseries \itshape \scshape \sfserief etc...\itfam
{}% measure of space to indent
{\bfseries}% name of head font
{}% punctuation between head and body
{0.5 em}% space after theorem head; " " = normal interword space
{\thmname{#1}\thmnumber{ #2}\normalfont\thmnote{ (#3)}\bfseries.}% Manually specify head

\newtheoremstyle{mystyle2}% name of the style to be used
{}% measure of space to leave above the theorem. E.g.: 3pt
{}% measure of space to leave below the theorem. E.g.: 3pt
{\itshape}% name of font to use in the body of the theorem \bfseries \itshape \scshape \sfserief etc...\itfam
{}% measure of space to indent
{\bfseries}% name of head font
{}% punctuation between head and body
{0.5 em}% space after theorem head; " " = normal interword space
{\thmname{#1}\thmnumber{ #2}\normalfont\thmnote{ (#3)}\bfseries.}% Manually specify head 

\theoremstyle{mystyle2}
\newtheorem{thm}{Theorem}[section]
\newtheorem*{thm*}{Theorem}

\newtheorem{prop}[thm]{Proposition}
\crefname{prop}{Proposition}{Propositions}
\newtheorem{lem}[thm]{Lemma}
\crefname{lem}{Lemma}{Lemmas}
\newtheorem*{lem*}{Lemma}
\newtheorem{kor}[thm]{Corollary}

\theoremstyle{mystyle}
\newtheorem{defi}[thm]{Definition}
\newtheorem*{defi*}{Definition}
\crefname{defi}{Definition}{Definitions}
\newtheorem{examp}[thm]{Example}
\newtheorem*{examp*}{Example}
\crefname{examp}{Example}{Examples}
\newtheorem{rem}[thm]{Remark}
\crefname{rem}{Remark}{Remarks}

\crefname{case}{Case}{Cases}

\setlist[enumerate,1]{label = \normalfont(\arabic*),
	ref = \thethm.(\arabic*)}

\DeclareMathOperator{\id}{id}
\DeclareMathOperator{\res}{res}

\DeclareMathOperator{\ac}{ac}

\newcommand{\emptydefinition}{\emptyset}
\renewcommand{\emptyset}{\varnothing}

\renewcommand{\vec}[1]{\underline{#1}}

\newcommand{\map}[1]{\omega_{#1}}

\newcommand{\Z}{\mathbb{Z}}
\newcommand{\N}{\mathbb{N}}

\newcommand{\Q}{\mathbb{Q}}

\newcommand{\bC}{\mathbb{C}}

\DeclareMathOperator{\Frob}{Frob}
\DeclareMathOperator{\AC}{AC}

\newcommand{\cT}{\mathcal{T}}

\renewcommand{\L}{\mathcal{L}}
\newcommand{\Lring}{\L_{\mathrm{ring}}}
\newcommand{\Loag}{\L_{\mathrm{oag}}}

\newcommand{\Lval}{\L_{\mathrm{val}}}
\newcommand{\Ltval}{\L_{\mathrm{compval}}}

\newcommand{\LWitt}{\L_{\mathrm{Witt}}}

\renewcommand{\O}{\mathcal{O}}
\newcommand{\m}{\mathfrak{m}}

\newcommand{\LQE}{\L_\mathrm{QE}}
\newcommand{\Lold}{\L_\mathrm{old}}

\newcommand{\VF}{\mathrm{VF}}
\newcommand{\VG}{\mathrm{VG}}

\numberwithin{equation}{section}

\begin{document}
\begin{abstract}
    We show that for any henselian valued field of mixed characteristic $(0,p)$,
    the (existential) $\Lval$-theory of the valued field is determined by
    the (existential) theory of the value group in $\Loag$ with a constant for $v(p)$
    and the (existential) theory of the residue ring $\O_v/(p)$ in an expansion
    $\LWitt$ of the language of rings, provided $\O_v/(p)$ is semiperfect.
    We moreover show that the $\LWitt$-structure on $\O_v/(p)$ is $\Lring$-definable
    using constants, and that this is exactly the structure induced on $\O_v/(p)$
    by the ambient valued field. 
    
    As a consequence, we obtain a relative quantifier 
    elimination result (eliminating $K$-quantifiers) in a suitable language 
    for the theory of roughly deeply ramified henselian valued fields
    of mixed characteristic $(0,p)$.
\end{abstract}
    \maketitle
    
    \section{Introduction and Motivation}
    The aim of this paper is to prove Ax--Kochen/Ershov type theorems for the class of henselian
    valued fields $(K,v)$ of mixed characteristic $(0,p)$ with semiperfect residue ring 
    $\O_v/(p)$, reducing questions about $(K,v)$ -- or, more specifically, its (existential) $\Lval$-theory -- 
    down to the corresponding
    questions for $\O_v/(p)$ and $(vK,v(p))$. We develop a structure theory
    which can be seen as a relative quantifier elimination
    statement, and which requires to consider $\O_v/(p)$ as a structure in an expansion 
    of the $\Lring$-language which we call $\LWitt$.

    Ax--Kochen/Ershov type theorems have a long history, starting with the work of
    Ax and Kochen \cite{AK65I} and,
    independently, Ershov \cite{ershovElementaryTheoryMaximal1965} who proved that the theory of a henselian
    valued field of residue characteristic $0$ (in the language
    $\mathcal{L}_\mathrm{val}$ of valued fields)\footnote{Notations and terminology are
    explained in \cref{sec:not}.} is entirely determined by the theories of its
    residue field (in the language $\mathcal{L}_\mathrm{ring}$ of rings) and its value group (in the language $\mathcal{L}_\mathrm{oag}$ of ordered abelian
    groups). When $(K,v)$ is henselian of mixed characteristic $(0,p)$ such that $v(p)$
    is minimum positive in the value group, then analogous results hold (cf.~\cite{AK65II}
    for the case of perfect residue fields and \cite{anscombe-jahnke2022cohen} for imperfect ones). More generally, van den Dries establishes that if $(K,v)$ and $(L,w)$ are 
    henselian valued fields
    of mixed characteristic, then
    \begin{align}
        \underbrace{(K,v) \equiv (L,w)}_{\text{in }\mathcal{L}_\mathrm{val}} \Longleftrightarrow \underbrace{\mathcal{O}_v/(p^n) \equiv \mathcal{O}_w/(p^n)}_{\text{in }\mathcal{L}_{\mathrm{ring}}}
        \textrm{ for all }n \in \mathbb{N} \textrm{ and } \underbrace{vK \equiv wL}_{\text{in }\mathcal{L}_{\mathrm{oag}}} \tag{$\star$}\label{eq:AKE}
    \end{align}
    holds (see \cite{Weis} for the case of certain value groups, and \cite[p. 144]{vdd} for the general case).
    In the special case where the value group interval $[0,v(p))$ is finite (i.e., the valued field
    is finitely ramified) and the residue
    field is perfect, 
    Lee and Lee show in \cite{LeeLee} that a single residue ring $\mathcal{O}/(p^N)$ suffices when \(N\) is chosen sufficiently large.
    In \cite{anscombeAxKochenErshov2024}, Anscombe, Dittmann and Jahnke show that 
    for henselian finitely ramified fields of fixed ramification index $e=|[0,v(p))|$, the
    complete (existential) theory of the valued field can be reduced to 
    the theory of the value group in $\mathcal{L}_\mathrm{oag}$ and the theory
    of the residue field \emph{in an expansion} $\mathcal{L}_{p,e}$ of the language of rings.
    They also show that expanding the language on the residue field is necessary in full generality, and that the $\mathcal{L}_{p,e}$-structure is $\Lring$-definable with parameters
    on the residue field and is exactly the structure induced on the residue field in the
    $\mathcal{L}_\mathrm{val}$-structure $(K,v)$.

    In \cite{jahnke-kartas2023beyond}, two AKE-type principles are proven for henselian
    valued fields of mixed characteristic $(0,p)$ with semiperfect residue ring $\O/(p)$
    when the value of $p$ is not minimum positive, i.e., 
    roughly deeply ramified henselian
    fields of mixed characteristic (introduced in \cite{KuhlBla}). More precisely, if $(K,v) \subseteq (L,w)$
    are two such fields, then one has
    $$
     \underbrace{(K,v) \preceq (L,w)}_{\text{in }\mathcal{L}_\mathrm{val}} \Longleftrightarrow \underbrace{\mathcal{O}_v/(p) \preceq \mathcal{O}_w/(p)}_{\text{in }\mathcal{L}_{\mathrm{ring}}}\textrm{ and } \underbrace{vK \preceq wL}_{\text{in }\mathcal{L}_{\mathrm{oag}}}
    $$
    (this is a special case of \cite[Theorem 5.1.4]{jahnke-kartas2023beyond}) and 
     $$
     \underbrace{(K,v) \preceq_\exists (L,w)}_{\text{in }\mathcal{L}_\mathrm{val}} \Longleftrightarrow \underbrace{\mathcal{O}_v/(p) \preceq_\exists \mathcal{O}_w/(p)}_{\text{in }\mathcal{L}_{\mathrm{ring}}}\textrm{ and } \underbrace{vK \preceq_\exists wL}_{\text{in }\mathcal{L}_{\mathrm{oag}}},$$
    see \cite[Theorem 5.1.10]{jahnke-kartas2023beyond}. These theorems apply in particular to perfectoid fields, whose model-theoretic properties (in continuous
    logic) are also the subject of a recent paper by Rideau-Kikuchi, Scanlon, and Simon \cite{RSS}.

    In this paper, we now show corresponding principles for 
    elementary equivalence and existential equivalence, i.e., for
    $\equiv$ and $\equiv_\exists$.   Akin to the situation in \cite{anscombeAxKochenErshov2024}, we need to add
    structure on $\O/(p)$ in order to achieve this (cf.~\cref{ex:structure}). Thus, we
    consider the residue rings as $\mathcal{L}_\mathrm{Witt}$-structures, a
    language explained in \cref{sec:structure}. Our main theorem is the
    following:
    \begin{thm*}[\cref{thm:main}]
        Let $(K,v)$ and $(L,w)$ be henselian valued fields of mixed characteristic 
        $(0,p)$
        such that $\O_v/(p)$ and $\O_w/(p)$ are semiperfect. Then, we have
    $$\underbrace{(K,v) \equiv (L,w)}_{\text{in }\mathcal{L}_\mathrm{val}} \Longleftrightarrow \underbrace{\mathcal{O}_v/(p) \equiv \mathcal{O}_w/(p)}_{\text{in }\mathcal{L}_{\mathrm{Witt}}}\textrm{ and } \underbrace{(vK,v(p)) 
    \equiv (wL,w(p))}_{\text{in }\mathcal{L}_{\mathrm{oag}}\cup\{c\}},$$
        and 
        $$\underbrace{(K,v) \equiv_\exists (L,w)}_{\text{in }\mathcal{L}_\mathrm{val}} \Longleftrightarrow \underbrace{\mathcal{O}_v/(p) \equiv_\exists \mathcal{O}_w/(p)}_{\text{in }\mathcal{L}_{\mathrm{Witt}}}\textrm{ and } \underbrace{(vK,v(p)) \equiv_\exists (wL,w(p))}_{\text{in }\mathcal{L}_{\mathrm{oag}}\cup\{c\}},$$
        where $c$ is a constant symbol interpreted as the value of $p$.
        Moreover, when $vK$ and $wL$ are regular, the value group data on the right hand side can be omitted in both equivalences.
    \end{thm*}
    Recall that an ordered abelian group is called regular if and only if it is
    $\Loag$-elementarily equivalent to a subgroup of $(\mathbb{R},+,<)$, see
    \cite{Rob}.
   Since we show in \Cref{prop:ring} that the $\mathcal{L}_\mathrm{Witt}$-structure is $\Lring$-definable on $\O_v/(p)$ using parameters, it makes sense that 
    the extra structure
    plays no role for the principles involving $\preceq$ and $\preceq_\exists$ proven
    in \cite{jahnke-kartas2023beyond}.
    
    A particularly well-behaved class of deeply ramified fields is that of tame valued
    fields, which were introduced and studied by Kuhlmann \cite{kuhlmann2016algebra}.
    While for tame fields of positive characteristic, both $\mathrm{AKE}_\equiv$ and 
    $\mathrm{AKE}_{\equiv_\exists}$ principles hold down to residue field
    and value group (see \cite[Theorem 7.1]{kuhlmann2016algebra}), both fail in full generality in mixed
    characteristic by \cite[Theorem 1.5(b)]{anscombe-kuhlmann2016notes}, and
    even after fixing the algebraic part
     and the value of $p$  (see both \cite{ketelsenCompositionAxKochenErshov2026} and its appendix by Dittmann).
    In fact, even for mixed characteristic tame fields $(K,v)$, adding 
    further structure
    on $\mathcal{O}_v/(p)$ is necessary (see \cref{ex:tame}). 
    In particular, we prove new results also for the class
    of tame fields of mixed characteristic $(0,p)$.

    Moreover, we deduce a relative quantifier elimination result in roughly deeply ramified 
    henselian valued fields from our Main Theorem (\cref{thm:main}).
    This quantifier elimination is in fact a syntactic consequence of our structure 
    theory from \Cref{ssec:structure.theory}, but we feel it is worth recording: even for tame fields of mixed characteristic, it is the first instance of an elimination of field quantifiers, 
    down to sorts for the value group and residue ring. More precisely,
    we define a language $\mathcal{L}_\mathrm{QE}$ with sorts $\VF$ for the field,
    $R_1$ for the residue ring $\O/(p)$ and $\VG$ for the value group, and
    maps $v\colon \VF\to \VG$ and $\res_1\colon \VF\to R_1$,
    as well as predicates $\operatorname{AC}_n\subseteq \VF\times R_1^n$ for $n\geq 1$, and
     predicates $S_n\subseteq \VG\times R_1^n$ which are interpreted to reflect the
     $\LWitt$-structure.
    Our quantifier elimination result then reads as follows:
    \begin{thm*}[{{{Theorem~\ref{thm:qe}}}}]
        Let $\cT$ be the theory of roughly deeply ramified henselian valued fields of
        mixed characteristic $(0,p)$
        with a compatible system of angular component maps.
        Then $\cT$ eliminates \(\VF\)-quantifiers in the language $\LQE$.
    \end{thm*}
    The fact that one can eliminate field quantifiers in 
    henselian valued fields of mixed
    characteristic in an $\omega$-sorted language with sorts for value group
    and all residue rings can be deduced from the general AKE-type statement 
    \eqref{eq:AKE} discussed above (see \cite[Theorem 4.2]{Weis} for the case allowing only certain value groups and \cite[Theorem 3.3.10]{rideau2022model} for the general case). In fact, 
    our QE result and its proof are very close
    in spirit to that of \cite[Theorem 3.3.10]{rideau2022model}.

\section{Notation, Terminology, and Guiding Examples} \label{sec:not}
Given a valued field \( (K,v) \), we denote its value group by \(vK\), its valuation ring by \(\O_v\), with maximal ideal \(\mathfrak{m}_v\), and its residue field by \(Kv\). We denote the residue map \(\O_v\to Kv=\O_v/\m_v\) by \(\res_v\).
If $(K,v)$ has mixed characteristic $(0,p)$, we call the quotient ring 
$\O_v/p^n\O_v=:\O_v/(p^n)$ the \emph{$n$th residue ring} and often refer to $\O_v/(p)$ as the \emph{residue ring}. 
We say that $\O_v/(p)$ is \emph{semiperfect} if the Frobenius map $\Frob\colon x \mapsto x^p$ is surjective.

Unless stated otherwise, we consider valued fields 
in the three-sorted language $\mathcal{L}_\mathrm{val}$.
As usual, this consists of 
a sort $K$ for the field,
a sort $\Gamma\cup\{\infty\}$ for the value group $vK$ together with infinity,
and a sort $k$ for the residue field $Kv$.
The two field sorts each carry the language of rings
$\mathcal{L}_\mathrm{ring} = \{+,\cdot,-,0,1\}$ 
(where $+$ and $\cdot$ are binary and $-$ is unary),
and the value group sort carries $\mathcal{L}_\mathrm{oag} = \{0, +, <\}$.
In addition, we have symbols for
the valuation map $v \colon K \twoheadrightarrow \Gamma \cup \{ \infty \}$ and the residue map
$\res_v \colon K \twoheadrightarrow k$, which we interpret as the constant zero map outside the valuation ring $\mathcal{O}_v$. Note that the $\Lring$-structure on
$\O_v/(p)$ is interpretable without parameters in $\Lval$.

Following Kuhlmann and Rzepka \cite{KuhlBla}, we call a valued field
$(K,v)$
of mixed characteristic $(0,p)$ \emph{roughly deeply ramified} if $\mathcal{O}_v/(p)$
is semiperfect (i.e., the map $x \mapsto x^p$ is surjective) and $v(p)$ is not the smallest positive element of the value group. In particular, any deeply ramified
valued field (in the sense of \cite{GR}) of mixed characteristic is roughly deeply ramified (see \cite[Introduction]{KuhlBla} for details).

We now give an example that in order to prove AKE-type theorems for henselian
valued fields of mixed characteristic $(0,p)$ with $\O/(p)$ semiperfect
down to $\O/(p)$ and the pointed value group, extra structure is required:
\begin{examp} \label{ex:structure}
    Let $(K,v)$ be the completion of $ {\big(\mathbb{Q}_p(\sqrt[p-1]{p},\sqrt[p^\infty]{p}),v_p\big)}$ 
    and 
    $(L,w)$ be the completion of ${\big(\mathbb{Q}_p(\mu_{p^\infty}),v_p\big)}$. Then, $(K,v)$ and $(L,w)$ are well-known 
    examples of perfectoid fields
    whose tilts are both isomorphic to
    the completed perfect hull $(F,u)$ of $(\mathbb{F}_p(\!(t)\!),v_t)$.
    In fact, we have
    \begin{align*}\O_v/(p) & \overset{{p} \mapsto t}{\cong} \O_u/(t^{p-1}) \cong \mathbb{F}_p[t^{1/p^\infty}]/(t^{p-1}) \overset{t+1 \mapsto \mu_p}{\cong} 
    \mathcal{O}_w/(p)
    \end{align*}
    and $$(vK,v(p)) \cong (\mathbb{Z}[1/p],p-1)\cong (wL,w(p)).$$
    In particular,
    $$\O_v/(p) \equiv \O_w/(p) \textrm{ and }(vK,v(p))\equiv (wL,w(p))$$
    holds.
    However, as $\sqrt[p]{p} \in K$ and $\sqrt[p]{p} \notin L$ we obtain 
    $$\underbrace{(K,v) \not\equiv (L,w).}_{\textrm{ in }\Lval \textrm{ (or even $\Lring$)}}$$
    Thus, it is necessary to add further structure on $\O_v/(p)$ and/or $(vK,v(p))$
    to have any hope of an AKE-type principle. Moreover, as in this example 
    $vK$ and $wL$ are both
    regular, the $\Lring$-theory of $\O_v/(p)$ (resp.~$\O_w/(p)$) completely
    determines the $\Loag\cup\{c\}$-theory of $(vK,v(p))$ (respectively that of $(wL,w(p))$), see the proof of \cite[Theorem 5.1.2]{jahnke-kartas2023beyond}. 
    This motivates adding further structure on $\mathcal{O}/(p)$.
\end{examp}

    One could hope that when one considers tame fields of mixed characteristic,
    no further structure on $\O/(p)$ might be required. This is however not the case:
    \begin{examp}[Dittmann] \label{ex:tame}
        Let $p\neq 2$ and $(K_0,v_0)$ be an algebraic 
        extension of \((\mathbb{Q}_p,v_p)\) with residue field
        $\mathbb{F}_p$ and value group $\mathbb{Z}[1/p]$ which is tame, and let $$(K,v) \coloneqq \widehat{(K_0,v_0)}$$ denote its completion. 
        Take any $\alpha \in \mathbb{F}_p^\times\setminus (\mathbb{F}_p^\times)^2$,
        and let $a \in \mathbb{Z}$ with $a + p\mathbb{Z} = \alpha$.
        Consider the extensions $(K_1,v_1)=(K(\sqrt{p}),v)$ and 
        $(K_2,v_2)=(K(\sqrt{ap}),v)$ (where by an abuse of notation we also use $v$ to denote the unique extension of $v$ to the finite extensions of $K$). 
        Note that both
        $(K_1,v_1)$ and $(K_2,v_2)$ are tame and perfectoid.
        Moreover we have $(K_1,v) \not\equiv (K_2,v)$ since $K_1 \ni \sqrt{p} \notin K_2$
        as $\alpha \notin (\mathbb{F}_p^\times)^2$.

        We claim that $\mathcal{O}_{v_1}/(p) \cong \mathcal{O}_{v_2}/(p)$ holds.
        Indeed, we have $\mathcal{O}_v/(p) \cong \mathcal{O}_{v^\flat}/(t)$
        for some $t \in \mathfrak{m}_{v^\flat}$ with $v^\flat(t)=v(p)$,
        and $$\mathcal{O}_{v_i^\flat}/(t) \cong \mathcal{O}_{v_i^\flat}/(\alpha t) \cong \mathcal{O}_{v_i}/(p)$$
        for $i \in \{1,2\}$ (since $v_i^\flat(\alpha t)=v_i^\flat(t)=v_i(p)$ as $\alpha \in ({\O_{v_i^\flat}})^\times$). However, as we have 
        $$
        (K_1,v_1)^\flat \overset{t \mapsto \alpha t}{\cong} (K_2,v_2)^\flat,
        $$
        we obtain 
        $$
        \mathcal{O}_{v_1}/(p) \cong \mathcal{O}_{v_1^\flat}/(t) \cong \mathcal{O}_{v_2^\flat}/(\alpha t) \cong \mathcal{O}_{v_2}/(p)
        $$
        as desired. Thus, even for a tame field $(L,w)$, further structure on 
        $\mathcal{O}_w/(p)$
        is required in order to determine $\mathrm{Th}_{\Lval}(L,w)$ entirely.
    \end{examp}

    The $\LWitt$-structure we add on $\O_v/(p)$ in the next section will allow us in particular to recover the higher residue rings
    $\O_v/(p^n)$ (for every $n \geq 1$) from $\O_v/(p)$. As mentioned in the introduction, in the
    case of henselian finitely ramified fields, the theory of $(K,v)$ reduces to those of
    $vK$ and $\O_v/(p^N)$ for some sufficiently large $N$ by the work of Lee and Lee \cite{LeeLee}. In general, this is no longer true:

    \begin{examp}[Kartas]\label{ex:kartas}
        For any $\alpha \in 2^\omega$, Kartas constructs an algebraic extension 
        $K_\alpha \supseteq \mathbb{Q}_p$ such that $K_\alpha\not\equiv_{\Lring} K_\beta$ for any $\beta \in 2^\omega$ with $\alpha \neq \beta$ 
        (see \cite[Claim 2 in the proof of Proposition 3.6.9(b)]{Kar24}). 
        More precisely, $K_\alpha$ is inductively obtained as follows:
        Starting with $K_0=\mathbb{Q}_p$ and $\pi_0=p$, define for $n>0$
        $$\pi_{\alpha\restriction n}=((1+p)^{\alpha(n-1)} \pi_{\alpha\restriction (n-1)})^{1/p} \textrm{ and } K_{\alpha \restriction n}=K_{\alpha \restriction {(n-1)}}( \pi_{\alpha\restriction n}).$$ 
        Then set $K_\alpha:=\bigcup_{n \in \omega} K_{\alpha\restriction n}$.
        When $K_\alpha$ is equipped with the
        unique prolongation $v_\alpha$ of the $p$-adic valuation, then $(K_\alpha,v_\alpha)$ is
        deeply ramified \cite[Claim 3 in the proof of Proposition 3.6.9(b)]{Kar24}.
        Moreover, Kartas argues in the proof of \cite[Proposition 7.1.1]{Kar24}, that for
        any distinct
        $\alpha, \beta \in 2^\omega$ 
        such that $n\in \omega$ is minimal with $\alpha(n)\neq \beta(n)$
        one gets $\O_{v_\alpha}/(p^{n+1}) \cong \O_{v_\beta}/(p^{n+1})$
        but $\O_{v_\alpha}/(p^{n+2}) \not \cong \O_{v_\beta}/(p^{n+2})$ and even
        $\O_{v_\alpha}/(p^{n+2}) \not \equiv_{\Lring} \O_{v_\beta}/(p^{n+2})$. 
        Thus, there is no bound $N >0$ such that for all (roughly) deeply ramified
        henselian valued fields of mixed characteristic the $\Lring$-theory of the 
        $N$th residue ring $\O/(p^N)$
        already determines those of $\O/(p^m)$ for $m \geq N$.
    \end{examp}
    
    \section{Recovering the valued field from residue ring and value group}

    \subsection{The Witt-structure on \(\O_v/(p)\)}
    \label{sec:structure}
    Fix a prime \(p\).
    Given \(n\in \N\), we denote the \(n\)th Witt polynomial by \(W_n\in \Z[X_0,\ldots,X_n]\), see \cite[Section~6.2]{vdd}:
    \begin{align*}
    	W_n\coloneqq \sum_{i=0}^n p^iX_i^{p^{n-i}}
    \end{align*}

    We make repeated use of the following well-known fact which follows from the binomial theorem:
    \begin{lem}[Warm-up Lemma, cf.~{\cite[\S 4.II, Lemme 1]{serre1968corps}}]
        \label{lem:warm.up}
        Let $(K,v)$ be a valued field of mixed characteristic $(0,p)$. 
        Then
        \[x\equiv y \mod p\O_v \Longrightarrow x^{p^n}\equiv y^{p^n}\mod p^{n+1}\O_v.\]
    \end{lem}
    Given a valued field $(K,v)$ of mixed characteristic $(0,p)$, we define for $n \geq 1$ the map
    \[
    \begin{array}{rrcl}
    \map{n}\colon &\left(\O_v/(p)\right)^n & \longrightarrow & \O_v/(p^n)\\
    &(x_0+p\O_v,\ldots,x_{n-1}+p\O_v) & \longmapsto & W_{n-1}(x_0,\ldots,x_{n-1})+p^n\O_v.
    \end{array}
    \]
    This is well-defined by \cref{lem:warm.up}.
    The kernel of this map is 
    \[
    \ker(\map{n})\coloneqq \left\{(\vec{x};\vec{y})\in \left(\O_v/(p)\right)^n\times \left(\O_v/(p)\right)^n : \map{n}(\vec{x})=\map{n}(\vec{y})\right\}.
    \]

    \begin{defi}[\(\LWitt\)]
        We define \(\LWitt\coloneqq \Lring \cup \{ \ker_n \colon n\in \N \}\), where \(\ker_n\) is a \(2n\)-ary relation symbol.

        Given a valued field \( (K,v) \) of mixed characteristic $(0,p)$ we define the \(\LWitt\)-structure on the ring \(\O_v/(p)\) as follows:
        \begin{itemize}
            \item the symbols in \(\Lring\) are interpreted by the usual ring-structure on \(\O_v/(p)\).
            \item for \(n\in \N\), \(\ker_n\) is interpreted by the set \(\ker(\map{n})\).
        \end{itemize}
    \end{defi}

    \begin{rem}
        Note that the $n$-ary relation $\ker_n(\cdot;\vec{0})$ $\emptyset$-defines the $2n$-ary relation $\ker_n$, and vice versa, see \cref{rem:ker0}. 
        For ease of notation, we 
        work with the \(2n\)-ary relation $\ker_n$ throughout. \label{rem:kern}
    \end{rem}
    
    Note: \(\ker_1\) is interpreted as the kernel of \(\map{1}=\id\colon \O_v/(p) \to \O_v/(p)\), which is just the equality relation.

    \subsection{Structure theory}\label{ssec:structure.theory}
    In this section, we show that the $\LWitt$-structure for $p$-adically complete valued fields of mixed characteristic with $\O/(p)$ semiperfect determines the 
    valued field up to isomorphism, as well as an appropriate version for embeddings.
    We then deduce our main results from these facts.
    
    The following lemma gives an explicit interpretation of the higher residue
    rings in the $\LWitt$-structure $\O/(p)$. Recall that an interpretation of one
    structure in another is called quantifier-free if the defining formula for each unnested atomic formula is quantifier-free \cite[Sec.~5.4(a)]{Hod}.
    A version of the following lemma is proven in case $\O/(p)$ is \emph{perfect} 
    \cite[Lemma 3.4.5]{Kar24}, where the higher Witt rings are interpretable in $\Lring$. We generalize this, working under the weaker
    assumption of semiperfectness where $\LWitt$ (rather than $\Lring$) is necessary (cf.~\cref{ex:kartas}). When $\O/(p)$ is
    perfect, then $\ker_n$ is just the equality relation, thus $\Lring$ suffices.
    \begin{lem}\label{lem:interpretability}
        Let \(p\) be a prime. Consider the class of roughly deeply ramified valued fields of mixed characteristic \( (0,p) \).
        Then \(\O/(p^n)\) (with \(\Lring\)) and the maps $\res_{n,m}\colon \O/(p^n)\to \O/(p^m)$ for $n \geq m \geq 1$ are uniformly quantifier-free interpretable in 
        \(\O/(p)\) (with \(\LWitt\)).
    \end{lem}

    \begin{proof}
        We show, for each roughly deeply ramified valued field \( (K,v) \) of mixed characteristic \((0,p)\), that \(\O_v/(p^n)\) (with \(\Lring\)) is interpretable in \(\O_v/(p)\) (with \(\LWitt\)), where the defining formulas do not depend on the choice of \( (K,v) \).

        We first claim that for any such field \( (K,v) \), 
        \[
        \begin{array}{rrcl}
        \map{n}\colon &\left(\O_v/(p)\right)^n & \longrightarrow & \O_v/(p^n)\\
        &(x_0+p\O_v,\ldots,x_{n-1}+p\O_v) & \longmapsto & W_{n-1}(x_0,\ldots,x_{n-1})+p^n\O_v.
        \end{array}
        \]
        is a well-defined surjective map.
        The fact that $\map{n}$ is well-defined follows from \Cref{lem:warm.up}.
        To see that $\map{n}$ is surjective, take $y\in \O_v$.
        By induction, we construct $x_0, \ldots, x_{n-1}\in \O_v$ such that 
        \[
        \map{n} (x_0 + p\O_v, \dotsc, x_{n-1}+ p\O_v) = y+p^n\O_v, \]
        more precisely, such that for $i= 0 \ldots, n$ we have
        \[
        y - x_0^{p^{n-1}} - px_1^{p^{n-2}} - \ldots - p^{i-1}x_{i-1}^{p^{n-i}}\in p^{i}\O_v.
        \]
 
        (For $i=0$ this statement is vacuous.)
        Suppose that $x_0, \ldots, x_{i-1}$ have been constructed, for some $i \in \{0, \ldots, n-1\}$.
        Write 
        \[
        y - x_0^{p^{n-1}} - px_1^{p^{n-2}} - \ldots - p^{i-1}x_{i-1}^{p^{n-i}} = p^{i}z
        \]
        for some $z\in \O_v$.
        Since $\O_v/(p)$ is semiperfect, there exists $x_{i}\in \O_v$ such that $x_{i}^{p^{n-1-i}} \equiv z\bmod p\O_v$, so
        \[
        y - x_0^{p^{n-1}} - px_1^{p^{n-2}} - \ldots - p^{i-1}x_{i-1}^{p^{n-i}}- p^{i}x_{i}^{p^{n-1-i}}\in p^{i+1}\O_v,
        \]
        as desired.
        
        We need to argue that for every unnested atomic formula in \(\Lring\), there is a quantifier-free defining formula in \(\LWitt\).
        In fact, it is enough to check this for \enquote{\(X=Y\)}, \enquote{\(X+Y=Z\)} and \enquote{\(X\cdot Y=Z\)}: one has \(W_n(0,\ldots,0)=0\), \(W_n(1,0,\ldots,0)=1\) and \(W_n(-1,0,\ldots,0)=-1\) if \(p\) is odd and \(W_n(1,\ldots,1)=-1\) if \(p=2\), see \cite[Exercise on page 135]{vdd}. This means that \(0\) and \(1\) are uniformly \( \emptydefinition \)-definable without quantifiers, and similarly for subtraction since it is multiplying by \(-1\).

        By definition of the \(\LWitt\)-structure on \(\O_v/(p)\), the defining formula for \enquote{\(X=Y\)} is just \(\ker_n(\vec{X};\vec{Y})\).

        By \cite[Lemma~6.5]{vdd}, there are \(\Z\)-polynomials \[S_0,S_1,\ldots,P_0,P_1,\ldots\in \Z[X_0,X_1,\ldots Y_0,Y_1,\ldots]\] (where for each \(i\) we have \(S_i,P_i\in \Z[X_0,\ldots,X_i,Y_0,\ldots,Y_i]\)) such that 
        the defining formula for \enquote{\(X+Y=Z\)} is given by 
        \[
        \ker_n(S_0(\vec{X},\vec{Y}),\ldots,S_{n-1}(\vec{X},\vec{Y});\vec{Z})
        \] 
        and the defining formula for \enquote{\(X\cdot Y=Z\)} is given by
        \[
        \ker_n(P_0(\vec{X},\vec{Y}),\ldots,P_{n-1}(\vec{X},\vec{Y});\vec{Z}).
        \]
        Note that the defining formulas just depend on \(n\) and \(p\) and not on the choice of \( (K,v) \).
        Thus \(\O_v/(p^n)\) is uniformly interpretable in the 
        $\LWitt$-structure \(\O_v/(p)\).

        To see that $\res_{n, m}\colon \O_v/(p^n)\to \O_v/(p^m)$ is $\LWitt$-interpretable for $n\geq m$, we simply note that
        \[
        \res_{n,m}(\omega_n(x_0, \ldots, x_{n-1})) = \omega_m(x_0^{p^{n-m}}, \ldots, x_{m-1}^{p^{n-m}})
        \]
        and the map $(\O_v/(p))^n\to (\O_v/(p))^m,\ (x_0, \ldots, x_{n-1})\mapsto (x_0^{p^{n-m}}, \ldots, x_{m-1}^{p^{n-m}})$ is clearly quantifier-free $\LWitt$-definable.
    \end{proof}

    \begin{rem}
        \begin{enumerate}
            \item \label[rem]{rem:kostas} Note that the proof of \cref{lem:interpretability} also goes through when only assuming that \(\O/(p)\) is semiperfect, i.e. without assuming that \(v(p)\) is not minimal positive. 
            Indeed, if \(v(p)\) is minimal positive, then \(\O/(p)\) is a perfect ring and our statement is exactly \cite[Lemma~3.4.5.]{Kar24} 
            where the simpler language without the \(\ker_n\) suffices, since $\ker_n$ is the identity. 
            \item \label[rem]{rem:kernenough} It follows immediately from the proof of \cref{lem:interpretability} that for fixed $n$, the reduct language $\Lring \cup \{\ker_1, \dotsc, \ker_n\}$ suffices to interpret quantifier-freely the rings 
            $\O_v/(p^m)$ and the maps $\res_{m,l}$ for $l\leq m \leq n$. 
            \item \label[rem]{rem:ker0} Using the Witt vector polynomials $S_i, P_i \in \mathbb{Z}[X_0, \dots, X_i, Y_0, \dots, Y_i]$, the $n$-ary relation 
            $\ker_n(\cdot; \vec{0}) \subseteq (\O_v/(p))^n$ $\emptyset$-defines the $2n$-ary
            relation $\ker_n$ (see \cref{rem:kern}). Indeed, for $(\vec{x},\vec{y})\in \left(\O_v/(p)\right)^n\times \left(\O_v/(p)\right)^n $, we have the equivalence
            $$\ker_n(\vec{x};\vec{y}) \Longleftrightarrow \ker_n(S_0(x_0,P_0(-1,y_0)),\dots,S_{n-1}(\vec{x},P_{n-1}(\vec{-1},\vec{y})); \vec{0})$$
            where we write $\vec{-1}$ for the $n$-tuple $(-1, 0, \dotsc, 0)$ (if $p\neq 2$) and $(1,1,\dotsc,1)$ if $p=2$.
        \end{enumerate}  
    \end{rem}
    
    We now give a structure theorem for 
    $p$-adically complete deeply ramified valued fields, i.e.\ deeply ramified fields
    $(K,v)$ with $\O_v \cong \varprojlim \O_v/(p^n)$.
    \begin{prop} \label{prop:lift}
        Let $(K,v)$ and $(L,w)$ be $p$-adically complete valued fields of mixed characteristic $(0,p)$
        such that $\O_v/(p)$ and $\O_w/(p)$ are semiperfect. 
        Then any embedding (respectively isomorphism)  
        $\O_v/(p) \to \O_w/(p)$ as $\LWitt$-structures is induced by a
        unique embedding (respectively
        isomorphism) $(K,v) \to (L,w).$
    \end{prop}
    \begin{proof}
        Let 
        \[
        \varphi\colon \left(\O_v/(p),\ldots\right) \hookrightarrow \left(\O_w/(p),\ldots\right)
        \]
        be an embedding of \(\LWitt\)-structures.
        For each \(n\in\N\), $\varphi$ induces a unique map \(\varphi_n\) as follows:
        \[
        \begin{array}{ccc}
            \varphi_n\colon \O_v/(p^n) & \hookrightarrow & \O_w/(p^n) \\
            \map{n}(\vec{x}) & \mapsto & \map{n}(\varphi(\vec{x}))
        \end{array}
        \]
        for any \(\vec{x}\in (\O_v/(p))^n\).
        As \(\varphi\) is an \(\LWitt\)-embedding, we have
        \[
        \ker_n(\varphi(\vec{x});\varphi(\vec{y})) \Leftrightarrow \ker_n(\vec{x};\vec{y}).
        \]
        Since \(\ker_n\) is interpreted as the kernel of \(\map{n}\), this shows that \(\varphi_n\) is well-defined and injective.
        
        Moreover, \(\varphi_n\) is a ring homomorphism.
        Indeed, since we have the uniform quantifier-free interpretation from \cref{lem:interpretability} (see also \cref{rem:kostas}), 
        there is a quantifier-free \(\LWitt\)-formula \(\chi_{+,n}\) such that
        \[
        \map{n}(\vec{x})+\map{n}(\vec{y})=\map{n}(\vec{z}) \text{ in \(\O_v/(p^n)\)} \iff \O_v/(p) \models \chi_{+,n}(\vec{x},\vec{y},\vec{z})
        \]
        for all \(\vec{x},\vec{y},\vec{z}\in (\O_v/(p))^n\), and 
        \[
        \map{n}(\vec{x})+\map{n}(\vec{y})=\map{n}(\vec{z}) \text{ in \(\O_w/(p^n)\)} \iff \O_w/(p) \models \chi_{+,n}(\vec{x},\vec{y},\vec{z})
        \]
        for all \(\vec{x},\vec{y},\vec{z}\in (\O_w/(p))^n\).
        Thus, if \(\vec{x},\vec{y}\in(\O_v/(p))^n\), and \(\map{n}(\vec{x})+\map{n}(\vec{y})=\map{n}(\vec{z}) \text{ in \(\O_v/(p^n)\)}\) for some \(\vec{z}\in (\O_v/(p))^n\), then by the above we have 
        \(\O_v/(p)\models \chi_{+,n}(\vec{x},\vec{y},\vec{z})\), and since \(\varphi\) is an \(\LWitt\)-embedding and \(\chi_{+,n}\) is quantifier-free, we get \(\O_w/(p)\models \chi_{+,n}(\varphi(\vec{x}),\varphi(\vec{y}),\varphi(\vec{z}))\).
        Now again by the above and applying the definition of \(\varphi_n\), we get
        \begin{align*}
        \varphi_n(\map{n}(\vec{x}))+\varphi_n(\map{n}(\vec{y}))&=\map{n}(\varphi(\vec{x}))+\map{n}(\varphi(\vec{y}))=\map{n}(\varphi(\vec{z}))\\&=\varphi_n(\map{n}(\vec{z}))=\varphi_n(\map{n}(\vec{x})+\map{n}(\vec{y}))
        \end{align*}
        Similarly, using the defining formulas for \enquote{\(X\cdot Y=Z\)}, \enquote{\(X=1\)}, and \enquote{\(\res_{n,m}(X)=Y\)}, one proves in the same way that 
        \begin{align*}
        	\varphi_n(\map{n}(\vec{x})) \cdot \varphi_n(\map{n}(\vec{y})) &= \varphi_n(\map{n}(\vec{x}) \cdot \map{n}(\vec{y})),\\
        	\varphi_n(1)&=1, \text{ and}\\
            \varphi_m(\res_{n,m}(\map{n}(\vec{x})))&=\res_{n,m}(\varphi_n(\map{n}(\vec{x}))).
        \end{align*}
        Thus, the \(\varphi_n\) are ring embeddings compatible with the inverse system maps, and by the universal property of inverse systems we obtain an embedding
        \[\Phi \colon \O_v\cong\varprojlim_n\O_v/(p^n)\hookrightarrow \varprojlim_n\O_w/(p^n)\cong\O_w.\]
        Note that if two $\LWitt$-embeddings $\varphi$ and $\psi$ induce the same
        maps $\varphi_n$ respectively $\psi_n$ for all $n$, then $\Phi$ 
        and $\Psi$ (the embedding of valuation rings induced by $\psi$) coincide by the universal property of inverse systems.
        
        Now $\Phi$ induces a unique embedding of fraction fields 
        $$K=\mathrm{Frac}(\O_v) \hookrightarrow L=\mathrm{Frac}(\O_w).$$ 
        Moreover, $\O_\nu:=\Phi(\O_v) \subseteq \Phi(K)$ is a refinement of 
        $\O_\omega = \O_w \cap \Phi(K)$. 
        Note that, for any refinement $\O_\nu \subseteq \O_\omega$, the induced map $\O_\nu/(p) \to \O_\omega/(p)$
        has kernel $$(p\O_\omega \cap \O_\nu)/p\O_\nu = p\O_\omega/p\O_\nu \cong \O_\omega/\O_\nu$$ where the first equality follows from 
        $p\O_\omega \subseteq \mathfrak{m}_\omega \subseteq \mathfrak{m}_\nu \subseteq  \O_\nu$. As the injectivity of the composition
        $\O_\nu/(p) \rightarrow \O_\omega/(p) \rightarrow \O_w/(p)$ implies 
        that of $\O_\nu/(p) \rightarrow \O_\omega/(p)$, we conclude
        $\O_\nu=\O_\omega$.
        
        Hence, $\Phi$ is in fact
        an embedding $\Phi\colon (K,v) \hookrightarrow (L,w)$ of valued fields,
        and is the unique such that induces $\varphi$.
        
        Finally, if $\varphi$ is an isomorphism, then \(\varphi\) and its inverse are $\LWitt$-embeddings. They lift to embeddings \(\Phi\) and \(\Phi'\) between \( (K,v) \) and \( (L,w) \) whose composition (in both ways) is a lift of the identity map of \(\O_v/(p)\) respectively \(\O_w/(p)\). 
        By uniqueness, there is only one lift of the identity map, namely the identity. 
        Thus, \(\Phi'=\Phi^{-1}\) and \(\Phi\) is an isomorphism.
    \end{proof}

    We can now prove our main theorem:
    \begin{thm}
        \label{thm:main}
        Let $(K,v)$ and $(L,w)$ be henselian valued fields of mixed characteristic 
        $(0,p)$ 
        such that $\O_v/(p)$ and $\O_w/(p)$ are semiperfect. Then, we have
        $$\underbrace{(K,v) \equiv (L,w)}_{\text{in }\mathcal{L}_\mathrm{val}} \Longleftrightarrow \underbrace{\mathcal{O}_v/(p) \equiv \mathcal{O}_w/(p)}_{\text{in }\mathcal{L}_{\mathrm{Witt}}}\textrm{ and } \underbrace{(vK,v(p)) 
        \equiv (wL,w(p))}_{\text{in }\mathcal{L}_{\mathrm{oag}}\cup\{c\}},$$
        and 
        $$\underbrace{(K,v) \equiv_\exists (L,w)}_{\text{in }\mathcal{L}_\mathrm{val}} \Longleftrightarrow \underbrace{\mathcal{O}_v/(p) \equiv_\exists \mathcal{O}_w/(p)}_{\text{in }\mathcal{L}_{\mathrm{Witt}}}\textrm{ and } \underbrace{(vK,v(p)) \equiv_\exists (wL,w(p))}_{\text{in }\mathcal{L}_{\mathrm{oag}}\cup\{c\}},$$
        where $c$ is a constant symbol interpreted as the value of $p$.
        Moreover, when $vK$ and $wL$ are regular, the value group data on the right hand side can be omitted in both equivalences.
    \end{thm}

    \begin{proof}
        If $(K,v)\equiv (L, w)$ then clearly we also have that $\O_v/(p) \equiv \O_w/(p)$ in $\LWitt$ and $(vK, v(p))\equiv (wL, w(p))$ in $\Loag\cup\{c\}$.
        Indeed, this follows from the fact that $\O_v/(p)$ and $(vK, v(p))$ with their respective languages are interpretable in $(K,v)$.
        If $\mathrm{Th}_\exists^{\Lval}{(K,v)} \subseteq \mathrm{Th}_\exists^{\Lval} (L,w)$, then -- potentially after replacing $(L,w)$ by an elementary extension -- we may assume that we are given an $\Lval$-embedding $\iota\colon K\to L$ (see \cite[Proposition 5.2.2]{CK}).
        It is clear that $\iota$ induces an $\Lring$-embedding $\O_v/(p)\to \O_w/(p)$ and an $\Loag\cup\{c\}$-embedding $(vK, v(p))\to (wL, w(p))$.
        We wish to show that the map $\O_v/(p)\to \O_w/(p)$ is in fact an $\LWitt$-embedding.
        For $x_0, \ldots, x_{n-1}, y_0, \ldots, y_{n-1}\in \O_v$ we have that
        \begin{align*}
        \ker_n(x_0+& p\O_v, \ldots, x_{n-1}+p\O_v; y_0+p\O_v, \ldots, y_{n-1}+p\O_v) \\
        &\Longleftrightarrow v\big(W_{n-1}(x_0, \ldots, x_{n-1}) - W_{n-1}(y_0, \ldots, y_{n-1})\big) \geq v(p^n).    
        \end{align*}
        Clearly this condition is preserved by $\iota$.
        Thus, we conclude $$\mathrm{Th}_\exists^{\LWitt}(\O_v/(p))\subseteq \mathrm{Th}_{\exists}^{\LWitt}(\O_w/(p)).$$
        By symmetry, we obtain
        $$\underbrace{(K,v) \equiv_\exists (L,w)}_{\text{in }\mathcal{L}_\mathrm{val}} \Longrightarrow \underbrace{\mathcal{O}_v/(p) \equiv_\exists \mathcal{O}_w/(p)}_{\text{in }\mathcal{L}_{\mathrm{Witt}}}\textrm{ and } \underbrace{(vK,v(p)) \equiv_\exists (wL,w(p))}_{\text{in }\mathcal{L}_{\mathrm{oag}}\cup\{c\}}.$$
        
        We now focus on the other direction. We assume that 
        $$
        \underbrace{\mathcal{O}_v/(p) \equiv_{(\exists)} \mathcal{O}_w/(p)}_{\text{in }\mathcal{L}_{\mathrm{Witt}}}\textrm{ and } \underbrace{(vK,v(p)) \equiv_{(\exists)} (wL,w(p))}_{\text{in }\mathcal{L}_{\mathrm{oag}}\cup\{c\}}.$$
        holds and -- without loss of generality -- further 
        that $(K,v)$ and $(L,w)$ are $\aleph_1$-saturated.
        In the $\equiv_\exists$-case, invoking \cite[Proposition 5.2.2]{CK} once again,
        we may assume there are embeddings
        \[
        f\colon \O_v/(p)\to \O_w/(p), \quad g\colon (vK, v(p))\to (wL, w(p))
        \]
        in the language $\LWitt$ resp.\ $\Loag\cup\{c\}$.
        In the $\equiv$-case, by the Keisler--Shelah theorem \cite[Theorem 6.1.15]{CK}, we may moreover assume that both
        $f$ and $g$ are isomorphisms.

        Let $\Delta_K\leq vK$ be the smallest convex subgroup of $vK$ containing $v(p)$, and similarly define $\Delta_L\leq wL$.
        Let $v_0\colon K^\times\to vK / \Delta_K$ and $w_0\colon L^\times\to wL / \Delta_L$ be the corresponding coarsenings of the valuations.
        By saturation, the residue field of $v_0$ is 
        \[K_0\cong \mathrm{Frac}(\varprojlim \O_v/(p^n))\]
        with induced valuation $\overline{v}\colon K_0^\times\to \Delta_K$, and similarly for the residue field $L_0$ of $w_0$.
        In other words, $K_0$ and $L_0$ are $p$-adically complete and both $\O_{\overline{v}}/(p) = \O_v/(p)$ and $\O_{\overline{w}}/(p) = \O_w/(p)$
         are  semiperfect.
        Hence by \Cref{prop:lift}, the map $f\colon \O_v/(p)\to \O_w/(p)$ lifts to give an isomorphism (resp.\ embedding)
        \[
        f'\colon (K_0, \overline{v})\to (L_0, \overline{w}).
        \]

        Since $g\colon (vK, v(p))\to (wL, w(p))$ maps $v(p)$ to $w(p)$, it descends to an isomorphism (resp.\ embedding)
        \[
        g'\colon vK/\Delta_K\to wL /\Delta_L. 
        \]

        We now make a case distinction:
        \begin{enumerate}
        \item In case both $f'$ and $g'$ are isomorphisms, the 
        Ax--Kochen--Ershov theorem in equicharacteristic zero shows 
        \[
         (K,v_0)\equiv (L, w_0).
        \]
        By resplendency (see \cite[Proposition 2.3]{ketelsenCompositionAxKochenErshov2026}), this implies 
        $(K,v)\equiv (L,w)$, as desired.
        
        \item In case $f'$ or $g'$ is not an isomorphism (but only an embedding), consider an $|K|^+$-saturated elementary extension $(M,\nu, \omega)$
        of the bi-valued field $(L,w_0,w)$. Then $f'$ and $g'$ give rise to embeddings
        $$f'' \colon (K_0,\overline{v}) \hookrightarrow (M\nu, \overline{\omega})
        \textrm{ and } g'' \colon vK/\Delta_K = v_0K \hookrightarrow \nu M$$
        where $\overline{\omega}$ is the valuation induced by \(\omega\) on $M\nu$.
        By the Embedding Lemma for henselian valued fields of equicharacteristic $0$, see \cite[Lemma 4.6.2]{PD}, we obtain an embedding $h: (K,v_0) \hookrightarrow (M,\nu)$ inducing both $f''$ and $g''$. As for any $x \in K$, we have
        \begin{align*}
        x \in \O_v &\Longleftrightarrow {\res}_{v_0}(x) \in \O_{\overline{v}} \Longleftrightarrow f'(\res_{w_0}(x)) \in \O_{\overline{w}} \\ & \Longleftrightarrow f''(\res_\nu (x)) \in \O_{\overline{\omega}} ,
        \end{align*}
        $h$ is automatically also an embedding of $(K,v)$ into $(M,\omega) \succeq (L,w)$. In particular, $\mathrm{Th}_\exists(K,v) \subseteq \mathrm{Th}_\exists(L,w)$.
        Symmetrically, starting with appropriate 
        embeddings of $\O_w/(p) \hookrightarrow \O_v/(p)$
        and $(wL,w(p)) \hookrightarrow (vK,v(p))$ we obtain $\mathrm{Th}_\exists(K,v) \supseteq \mathrm{Th}_\exists(L,w)$.
        \end{enumerate}

        For the moreover statement: $\aleph_1$-saturation shows that $vK/\Delta_K$ and $wL/\Delta_L$ are nontrivial.
        As regularity implies that $vK/\Delta_K$ and $wL/\Delta_L$ are divisible, see
        \cite[\S 3]{Rob}, we automatically have that $vK/\Delta_K \equiv_{(\exists)} wL / \Delta_L$.
    \end{proof}

    As another consequence of \cref{prop:lift}, we see that the residue ring $\O/(p)$ is stably embedded as an
    $\LWitt$-structure in any henselian roughly deeply ramified field of mixed 
    characteristic:
    \begin{kor} \label{kor:SE}
    In any henselian roughly deeply ramified field $(K,v)$ of mixed characteristic $(0,p)$, the imaginary
    $\mathcal{O}_v/(p)$ is stably embedded as an $\LWitt$-structure.
    \end{kor}
    \begin{proof}
        Let $(K,v)$ be an $\aleph_1$-saturated 
        henselian roughly deeply ramified field of mixed characteristic,
        and let $(L,w) \succeq (K,v)$ be a 
        monster model\footnote{If the reader feels uncomfortable about the use of monster models to show stable embeddedness, \cite[Section 3]{HK23} details how this can be avoided.} of $\mathrm{Th}_{\Lval}(K,v)$ (see \cite[Theorem 6.1.7]{TZ}).
        Take finite tuples $\vec{\alpha}, \vec{\beta} \in \O_w/(p)^m$ 
        satisfying 
        $${\LWitt}\textrm{-}\mathrm{tp}^{\O_w/(p)}(\vec{\alpha}\,|\,\O_v/(p)) = {\LWitt}\textrm{-}\mathrm{tp}^{\O_w/(p)}(\vec{\beta}\,|\,\O_v/(p)).$$ 
        We need to show that
        $${\Lval}\textrm{-}\mathrm{tp}^{(L,w)}(\vec{\alpha}\,|\,K) = \Lval\textrm{-}\mathrm{tp}^{(L,w)}(\vec{\beta}\,|\,K) $$
        holds.

        Since $(L,w)$ is a monster model of its theory, 
        so is the $\LWitt$-structure $\O_w/(p)$ (as it is interpretable). 
        By \cite[Corollary 6.1.8]{TZ}, there
        is an $\LWitt$-automorphism $\overline{\sigma}$ of $\O_w/(p)$ with $\overline{\sigma}(\vec{\alpha})=\vec{\beta}$ and $\overline{\sigma}(\delta)=\delta$ for any $\delta \in \O_v/(p)$ 
        (i.e., $\overline\sigma \in \mathrm{Aut}^{\LWitt}(\O_w/(p) \,|\, \O_v/(p))$).

        As in the proof of \cref{thm:main}, let $\Delta_K\leq vK$ be the smallest convex subgroup of $vK$ containing $v(p)$, and similarly define $\Delta_L\leq wL$.
        Let $v_0\colon K^\times\to vK / \Delta_K$ and $w_0\colon L^\times\to wL / \Delta_L$ be the corresponding coarsenings of the valuations with residue fields $K_0$ and $L_0$
        respectively. Let $\overline{v}$ (resp.~$\overline{w}$) denote the valuation on $K_0$ (resp.~$L_0$)
        induced by $v$ (resp.~$w$). 
        Note that we get $(K_0,\overline{v}) \subseteq (L_0,\overline{w})$ since \((K,v)\subseteq(L,w)\) and \(w_0|K=v_0\).
        Just like in \cref{thm:main}, \cref{prop:lift} implies that $\overline{\sigma}$ lifts to an $\Lval$-automorphism 
        $\sigma_0$ of
        $(L_0,\overline{w})$. As its restriction $\sigma_0|_{K_0}$
        is a lift of the identity map $\overline{\sigma}|_{\O_v/(p)}$, $\sigma_0$ fixes
        $K_0$ pointwise (this is the uniqueness statement in \cref{prop:lift}).
        Thus, choosing any preimage $\vec{a_0} \in (\O_{\overline{w}})^m$ of $\vec{\alpha} \in (\O_{\overline{w}}/(p))^m =
        (\O_w/(p))^m$,
        we get
        \begin{equation}
            {\Lval}\textrm{-}\mathrm{tp}^{(L_0,\overline{w})}(\vec{a_0}\,|\,K_0) = \Lval\textrm{-}\mathrm{tp}^{(L_0,\overline{w})}(\sigma_0(\vec{a_0})\,|\,K_0) \label{eq:res} 
        \end{equation}

        Now consider the enriched valued field $(L,w_0, \overline{w})$ in the language $\Ltval$, the expansion of $\Lval$ that adds $\overline{w}$ as additional structure on the residue field of $(L,w_0)$, and its $\Ltval$-substructure $(K,v_0,\overline{v})$. By the 
        resplendent version of the AKE Theorem in equicharacteristic $0$, see 
        \cite[Section 7.2]{vdd}, \eqref{eq:res} implies 
        \begin{equation}
        {\Ltval}\textrm{-}\mathrm{tp}^{(L,w_0,\overline{w})}(\vec{a_0}\,|\,K) = \Ltval\textrm{-}\mathrm{tp}^{(L,w_0,\overline{w})}(\sigma_0(\vec{a_0})\,|\,K) \label{eq2}
        \end{equation}
        As the $\Ltval$-structure $(L,w_0,\overline{w})$ interprets the $\Lval$-structure
        $(L,w)$, (\ref{eq2}) implies 
        \begin{equation}
        {\Lval}\textrm{-}\mathrm{tp}^{(L,w)}(\vec{a_0}\,|\,K) = \Lval\textrm{-}\mathrm{tp}^{(L,w)}(\sigma_0(\vec{a_0})\,|\,K).
        \end{equation}
        Finally, by reducing both $\vec{a_0}$ and $\sigma_0(\vec{a_0})$ modulo $p\O_{\overline{w}}$, we obtain 
        \begin{equation}
        {\Lval}\textrm{-}\mathrm{tp}^{(L,w)}(\vec{\alpha}\,|\,K) = \Lval\textrm{-}\mathrm{tp}^{(L,w)}(\vec{\beta} \,|\,K),
        \end{equation}
        as desired.
    \end{proof}

    In \cref{kor:SE.ring} we will prove that $\O_v/(p)$ is in fact stably embedded as an $\Lring$-structure.
    In contrast, the value group is not stably embedded as an $\Loag$-structure
    (which is the same as being stably embedded as an $\Loag \cup \{c\}$-structure): 

    \begin{examp}
    	Consider the example constructed in \cite[Example~6.10]{Kuh25}, namely a 
        deeply ramified field \( (L,v_p) \) of mixed characteristic, with a rank-1 and \( p \)-divisible value group, and with algebraically closed residue field \( Lv_p\eqqcolon K \) and a non-trivial valuation \( \overline v \) on \( K \). 
    	The composition \( (L,v\coloneqq \overline v \circ v_p) \) is a mixed characteristic deeply ramified field of rank 2 with \( p \)-divisible value group
        and there is a Galois degree-\( p \) defect extension \( (L(a)|L,v) \) with independent defect and associated convex subgroup \( \overline v K \).

    	Now, let $\tilde{v}$ be any prolongation of $v$ to $L^\mathrm{alg}$ and consider the ramification field \( (L^R,v^R) \) of \( (L^\mathrm{alg},\tilde{v})/(L,v) \). 
    	As \( vL \) is \( p \)-divisible, we get that \( \Gamma\coloneqq v^RL^R \) is divisible.
    	By \cite[Proposition~3.8]{KuhlBla}, \( (L^R(a)|L^R,v^R) \) is a degree-\( p \) independent defect extension where the associated convex subgroup corresponds to the rank-1 coarsening \( w \) of \( (L^R,v_R) \) (i.e., the unique coarsening of $v^R$ which is a prolongation of $v_p$).
    	
        By \cite[Theorem~4.11]{KRS}, $w$ is $\Lval(L^R)$-definable (even $\Lring(L^R)$-definable). Thus, the $\Lval(L^R)$-structure of $(L^R,v^R)$
        defines a non-trivial proper convex subgroup of the divisible ordered group $\Gamma$, which is not $\Loag(\Gamma)$-definable (e.g.\ by quantifier elimination~\cite[Corollary~3.1.17]{Mar}). In particular,
        $\Gamma$ is not stably embedded as an $\Loag$-structure in $(L^R,v^R)$.
    \end{examp}
    
    \subsection{Recovering the $\LWitt$-structure from $\Lring$ plus constants}

    We show that the $\LWitt$-structure on $\O_v/(p)$ is already $\Lring$-definable with parameters.
    This means that an appropriate version of Theorem~\ref{thm:main} for 
    $\preceq$ and $\preceq_{\exists}$ already holds with the language $\Lring$ on $\O_v/(p)$, as in~\cite[Thm.\,1.7.3]{jahnke-kartas2023beyond}.
    Recall that $\Frob\colon \O_v/(p)\to \O_v/(p)$ denotes the Frobenius morphism. 
    Note that for a non-zero element $\alpha$ in $\O_v/(p)$, the valuation $v(\alpha)$ is well-defined, and lies in $[0,v(p))\subseteq vK$.

    \begin{defi}
        For $n\geq 1$ define
        \[
        \theta_n\colon \ker(\Frob^{n})\to \O_v/(p^{n}),\ x + p\O_v\mapsto \frac{x^{p^{n}}}{p} + p^{n}\O_v.
        \]
    \end{defi}
    
    Note that $\theta_n$ is well defined.
    Indeed, if $x+p\O_v\in \O_v/(p)$ lies in the kernel of $\Frob^{n}$ then $x^{p^{n}} \in p\O_v$, so that $x^{p^{n}}/p$ is still an element of $\O_v$.
    Additionally, if $x,y\in \O_v$ are both representatives of $x+p\O_v$, then $x^{p^{n}} \equiv y^{p^{n}}\bmod p^{n+1}\O_v$, by Lemma~\ref{lem:warm.up}.
    
    \begin{prop} \label{prop:ring}
        In any roughly deeply ramified field $(K,v)$ of mixed characteristic $(0,p)$, 
        every $\LWitt$-definable subset of $(\O_v/(p))^n$ is $\Lring(\O_v/(p))$-definable.
    \end{prop}

    \begin{proof} 
        Consider the following statements
        \begin{enumerate}
            \item[($I_n$)] $\ker_n\subseteq (\O_v/(p))^{2n}$ is $\Lring(\O_v/(p))$-definable,
            \item[($D_n$)] the map $\theta_n\colon \ker(\Frob^{n})\to \O_v/(p^{n})$ is $\Lring(\O_v/(p))$-definable in $\O_v/(p)$.
        \end{enumerate}
        We will prove these statements via a joint induction on $n$, from which the result follows immediately.
        
        Note that if $(I_1), \ldots, (I_n)$ are true, then the rings $(\O_v/(p^m), +, \cdot)$ are $\Lring(\O_v/(p))$-interpretable in $\O_v/(p)$ for $m\leq n$.
        Indeed, this follows from the proof of \Cref{lem:interpretability} and \cref{rem:kernenough}, 
        which moreover gives the explicit interpretation
        \[
        \begin{array}{rrcl}
        \omega_n\colon& (\O_v / (p))^n &\to& \O_v / (p^n)\\ &(x_0+p\O_v, \ldots, x_{n-1}+p\O_v)&\mapsto &W_{n-1}(x_0, \ldots, x_{n-1})+p^n\O_v.
        \end{array}
        \]
        Hence the statement $(D_n)$ only makes sense once we know that $(I_n)$ is true.
        Additionally, the maps $\res_{m, k}\colon \O_v/(p^m)\to \O_v/(p^k)$ are also $\Lring(\O_v/(p))$-definable for $n\geq m\geq k$, since under the interpretation they are simply given by
        \[
        (\O_v/(p))^m\to (\O_v / (p))^k,\ (a_1, \ldots, a_m)\mapsto (a_1^{p^{m-k}}, a_2^{p^{m-k}}, \ldots, a_k^{p^{m-k}}).
        \]

        The statement $(I_1)$ clearly holds.
        Assume that $(I_1), \ldots, (I_n)$ and $(D_1), \ldots, (D_{n-1})$ are all true.
        Then we prove that also $(D_n)$ is true. 
        Take $a\in \O_v$ of valuation $v(p)/p^n$, which exists because $K$ is roughly deeply ramified, and consider the map
        \[
        \begin{array}{rrcl}
            g\colon  & \ker(\Frob^n) &\to & \O_v / (p^n), \\
            & b+p\O_v & \mapsto & \left( \frac{b}{a} \right)^{p^n} + p^n\O_v. 
        \end{array}
        \]
        We claim that $g$ only depends on $\alpha \coloneqq a + p \O_v$, 
        and that it is well-defined and $\Lring(\O_v/(p))$-definable in $\O_v/(p)$, which makes sense because of $(I_n)$.
        To see that $g$ is well-defined, we first note that for $x\in \O_v$ we have that
        \[
        (1+p^nx)^{-1} = 1 - {p^nx}(1+p^nx)^{-1}\in 1+p^n\O_v.
        \]
        In other words, $1+p^n\O_v$ is closed under inversion.
        Take $u, u'\in \O_v$. 
        By \cref{lem:warm.up} we can write $(b+pu)^{p^n} = b^{p^n} + p^{n+1}c$ for some $c\in \O_v$ and similarly, using that $v(a^{p^n}) = v(p)$, write $(a+pu')^{p^n} = a^{p^n}(1+p^nd)$ for some $d\in \O_v$.
        Then we have that
        \[
        \frac{(b+pu)^{p^n}}{(a+pu')^{p^n}} = \left(\frac{b^{p^n}}{a^{p^n}} + p^n c \frac{p}{a^{p^n}}\right)(1+p^n d)^{-1}\in\left( \frac{b}{a}\right)^{p^n} + p^n\O_v.
        \]
        Hence $g$ is well-defined and only depends on \( \alpha \).
        To see that $g$ is $\Lring(\O_v/(p))$-definable, we note that $g(\beta) = \gamma$ if and only if for some (equivalently any) $\delta\in \O_v/(p)$ with $\alpha \delta = \beta$ we have that $\gamma = \omega_n(\delta^p, 0, \ldots, 0)$.
        To conclude $(D_n)$, for $\beta\in \ker(\Frob^n)$ we then have
        \[
        \theta_n(\beta) = \theta_n(\alpha)g(\beta),
        \]
        so that $\theta_n$ is $\Lring(\O_v/(p))$-definable.
        In particular, the parameters needed to witness $(D_n)$ are those witnessing \( (I_1),\ldots,(I_n),(D_1),\ldots,(D_{n-1}) \) together with $\alpha$ and $\theta_n(\alpha)\in \O_v/(p^n)$.

        Next assume that $(I_1), \ldots, (I_n)$ and $(D_1), \ldots, (D_n)$ all hold.
        Then we prove that also $(I_{n+1})$ is true.
        In view of \cref{rem:ker0}, it is sufficient to show that $\ker_{n+1}(-; 0)$ is $\Lring(\O_v/(p))$-definable.
        So suppose that $\zeta = (\zeta_0, \ldots, \zeta_n)$ is an element of $(\O_v/(p))^{n+1}$ such that $\ker_{n+1}(\zeta; 0)$ holds.
        If $z_0, \ldots, z_n\in \O_v$ are lifts of $\zeta_0, \ldots, \zeta_n$, then by definition this means that $W_n(z_0, \ldots, z_n) \equiv 0 \bmod p^{n+1}\O_v$.
        Reducing this equation modulo $p$, we see that $\zeta_0^{p^n} = 0$, so that $\zeta_0\in \ker (\Frob^n)$.
        Dividing $W_n(z_0, \ldots, z_n)$ by $p$, we obtain that
        \begin{equation}\label{eq:params}
        \theta_n(\zeta_0) + \map{n}(\zeta_1,\ldots,\zeta_n)=0    
        \end{equation}
        in \( \O_v/(p^n) \).
        Conversely, if $\zeta\in (\O_v/(p))^{n+1}$ is such that $\zeta_0^{p^n} = 0$ and \eqref{eq:params} holds, then $\ker_{n+1}(\zeta; 0)$ holds.
        Because of $(I_n)$ and $(D_n)$, \eqref{eq:params} is an $\Lring(\O_v/(p))$-definable condition, and so we conclude that also $\ker_{n+1}$ is $\Lring(\O_v/(p))$-definable.
    \end{proof}

    As a consequence, one can now recover the AKE\({}_{\preceq}\) theorem from~\cite[Theorem~5.1.4]{jahnke-kartas2023beyond} by combining \cref{prop:ring} with the proof ideas of \cref{thm:main}.

    \begin{rem} 
    \begin{enumerate}
        \item Going through the previous proof, we see that to recover the $\LWitt$-structure it is enough to add constant symbols for elements $\{\alpha_n, \theta_n(\alpha_n)\}_{n\in \N}$, where $\alpha_n\in \O_v/(p)$ has valuation $v(p)/p^n$, and $\theta_n(\alpha_n)$ is considered as an element of $(\O_v/(p))^n$ via the interpretation $\omega_n$.
        For example, if $(K, v)$ is a perfectoid field of mixed characteristic $(0,p)$ and $p$ has a compatible system of $p$th power roots $(p^{1/p^n})_{n\in \N}$ then it is sufficient to add constants for $p^{1/p^n}+p\O_v\in \O_v/(p)$.
        Indeed, $p^{1/p^n}$ has valuation $v(p)/p^n$ and $\theta_n(p^{1/p^n}) = 1 \in \O_v/(p^n)$.
        \item
        \label[rem]{rem:claude}
        Let $(K,v)$ be a perfectoid field of mixed characteristic and
        take parameters $\{\alpha_n, \theta_n(\alpha_n)\}_{n\geq 1}$, where $\alpha_n\in \O_v/(p)$ and $\theta_n(\alpha_n)\in (\O_v/(p))^n$ as in (1) above.
        Assume in addition that  $\alpha_n^p=\alpha_{n-1}$ holds (which is always possible to arrange since $K$ is perfectoid)        
        and consider
\[
  \alpha \coloneqq (\alpha_n)_n \in \varprojlim_{x\mapsto x^p} \O_v/(p) = \O_v^\flat .
\]

Then
\[
  v^\flat(\alpha) = pv(\alpha_1) = 1 = v(p),
\]
i.e.\ $\alpha$ is a pseudouniformizer of the tilt.

By \cref{lem:warm.up} and the relation $\alpha_n^p = \alpha_{n-1}$, the elements
$\theta_n(\alpha_n)$ are compatible under the projections $\O_v/(p^{n+1}) \to \O_v/(p^{n})$, they therefore define an element
\[
  u \coloneqq \varprojlim_n \theta_n(\alpha_n) \in \varprojlim_n \O_v/(p^{n}) \cong \O_v .
\]
Since $v(\alpha_1) = 1/p$ we have $v(\theta_1(\alpha_1)) = 0$ and therefore $v(u) = v(\theta_1(\alpha_1)) = 0$.
In other words, $u\in \O_v^\times$.

Moreover, the two elements are related by the equation $\alpha^\sharp = pu$.
Writing $\theta\colon W(\O_v^\flat) \to \O_v$ for Fontaine's period map, we now 
choose any $\tilde{u} \in W(\O_v^\flat) $ with $\theta(\tilde{u})=u$. 
For $\beta\in \O_v^\flat$ denote by $[\beta]\in W(\O_v^\flat)$ the Teichmüller lift.
Then $[\alpha] -p\tilde{u}$ generates the kernel of $\theta$, and identifies $\O_v \cong  W(\O_v^\flat)/([\alpha] -p\tilde{u})$ 
as an untilt of
$\O_v^\flat$. 

Thus, choosing parameters (at least such satisfying the additional condition $\alpha_n^p=\alpha_{n-1}$ for all $n$) to define
the $\LWitt$-structure on $\O_v/(p)$ amounts to fixing $\O_v$ as an untilt of 
$\O_v^\flat$.
    \end{enumerate}
    \end{rem}

    \begin{kor}\label{kor:SE.ring}
        Let $(K,v)$ be a roughly deeply ramified henselian valued field of mixed characteristic $(0,p)$.
        Then $\O_v/(p)$ is stably embedded as an $\Lring$-structure.
    \end{kor}

    \begin{proof}
        This follows directly from \Cref{kor:SE} and \Cref{prop:ring}.
    \end{proof}

    \section{Quantifier elimination}

    Let $(K, v)$ be a valued field of mixed characteristic $(0,p)$.
    Recall that an \emph{angular component} is a multiplicative morphism $\ac_n\colon K^\times\to (\O_v/(p^n))^\times$ which agrees with the reduction map $\res_n\colon \O_v\to \O_v / (p^n)$ on $\O_v^\times$.
    We extend $\ac_n$ to $K\to \O_v/(p^n)$ by setting $\ac_n(0) = 0$.
    A system of angular components $(\ac_n)_n$ is said to be \emph{compatible} if for $n\geq m$ we have that $\ac_m = \res_{n,m}\circ \ac_n$.
    If $(K,v)$ is $\aleph_1$-saturated, then 
    a compatible system of angular components always exists:
    indeed, in that case there exists a cross-section $s \colon vK\to K^\times$ by~\cite[Corollary 5.5]{vdd} which may be used to define $\ac_n$ via $\ac_n(x) = \res_n(x s(-v(x)))$.
    For convenience, we use angular components in our relative quantifier elimination result. 
    It seems reasonable to expect that one can do without by using a language similar to the one in~\cite{ACGZ22}.

    Let $(K,v)$ be a roughly deeply ramified field of mixed characteristic $(0,p)$.
    In this section, we prove a relative quantifier elimination result down to the residue ring $\O_v/(p)$ and the value group $\Gamma$.
    For the most part, this is a syntactical reformulation of known quantifier elimination results, pulled back via the interpretation of $\O_v/(p^n)$ in $\O_v/(p)$ from \cref{lem:interpretability}.
    
    Let us first explain the language, which will be denoted by $\LQE$.
    The language is $3$-sorted, with the following sorts:
    \begin{enumerate}
        \item $\VF$ with the language of rings $\Lring$, 
        \item $\VG$ with the language of ordered abelian groups $\Loag$, and
        \item $R_1$ with the language $\Lring$.
    \end{enumerate}
    There are moreover symbols for
    \begin{enumerate}
        \item maps $v\colon \VF\to \VG$ and $\res_1\colon \VF\to R_1$,
        \item predicates $\AC_n\subseteq \VF\times R_1^n$ for $n\geq 1$, and
        \item predicates $S_n\subseteq \VG\times R_1^n$ for $n\geq 1$.
    \end{enumerate}

    If $(K,v)$ is a roughly deeply ramified valued field of mixed characteristic, and comes equipped with a system of compatible angular components $(\ac_n)_n$, then $(K,v)$ naturally has $\LQE$-structure.
    More precisely, $\VF$ is interpreted as $K$ with the ring operations, $\VG$ is interpreted as $\Gamma\cup\{\infty\}$ with the group operations, and $v\colon K\to \Gamma\cup\{\infty\}$ is the valuation.
    Let $\omega_n\colon (\O_v/(p))^n\to \O_v/(p^n)$ be the interpretation from above.
    The sort $R_1$ is interpreted as $\O_v/(p)$ with the $\Lring$-language.
    The predicate $\AC_n$ is the trace of $\ac_n\colon K\to \O_v/(p^n)$ on $\O_v/(p)$.
    More precisely, we have $(x,\xi)\in \AC_n$ if and only if $\ac_n(x) = \omega_n(\xi)$.
    We define a map $s_n\colon \Gamma_{\geq 0}\to \O_v/(p^n)$ by mapping an element $\gamma\in \Gamma_{\geq 0}$ to $\res_n(x)$, where $x\in K$ is any element with $\ac_n(x) = 1$ and $v(x) = \gamma$.
    Note that this map is well-defined and definable using $\ac_n$.
    Furthermore, $S_n$ is the trace of $s_n$ on $\Gamma \times R_1^n$.
    More precisely, we have $(\gamma, \xi)\in S_n$ if and only if $\gamma \geq 0$ and $s_n(\gamma) = \omega_n(\xi)$, or $\gamma < 0$ and $\xi = (0, \ldots, 0)$.

    Recall that if $t(\zeta_1, \ldots, \zeta_m)$ is an $\Lring$-term, then there exist $n$ $\Lring$-terms $t^W((\zeta_{i,j})_{i=1, \ldots, m, j=0, \ldots, n-1}) = (t^W_0, \ldots, t^W_{n-1})$ such that for $\zeta_{i,j}\in R_1$ for $i=1, \ldots, m$, $j=0, \ldots, n-1$ we have
    \[
    t(\map{n}(\zeta_{1,0}, \ldots, \zeta_{1,n-1}), \ldots, \map{n}(\zeta_{m,0}, \ldots, \zeta_{m,n-1})) = \map{n} (t^W((\zeta_{i,j})_{ij})).
    \]
    As in the proof of \Cref{lem:interpretability}, this follows from \cite[Lemma~6.5]{vdd}.
    For $n,m$ positive integers with $n\geq m$ we denote also by $\res_{n,m}\colon R_1^n\to R_1^m$ the map $\res_{n,m}\colon R_n\to R_m$ pulled back through the interpretation $\map{n}$.
    More explicitly, this is simply the map
    \[
    \res_{n,m}(\zeta_1, \ldots, \zeta_n) = (\zeta_1^{p^{n-m}}, \ldots, \zeta_m^{p^{n-m}}). 
    \]

    We let $\cT$ be the $\LQE$-theory of roughly deeply ramified henselian valued fields of mixed characteristic $(0,p)$ with a compatible system of angular component maps.
    Then we have the following.

    \begin{thm}
    \label{thm:qe}
        The theory $\cT$ eliminates $\VF$-quantifiers.
    \end{thm}

    \begin{proof}
        Even though $R_1$ does not carry the $\LWitt$-structure, it may be recovered using the symbols $S_n$.
        Indeed, for $\vec{x}\in R_1^n$ we have that $\ker_n(\vec{x}; \vec{0})$ if and only if 
        \[
            S_n(v(p^n), \vec{x}).
        \]
        Therefore, to prove relative quantifier elimination in $\LQE$, we may work with the language $\LWitt$ on $R_1$, instead of just $\Lring$.
    
        Let $\Lold$ be the $\omega$-sorted language with sorts $\VF$ with $\Lring$, $\VG$ with $\Loag$ and $R_n$ for $n\geq 1$ with $\Lring$.
        Additionally, we have maps $v\colon \VF\to \VG$, $\ac_n\colon \VF\to R_n$, $\res_n\colon \VF\to R_n$, $\res_{n,m}\colon R_n\to R_m$ and $s_n\colon \VG\to R_n$.
        Every model of $\cT$ naturally has $\Lold$-structure, where all symbols have their usual interpretation.
        In particular, $R_n$ is interpreted as $\O_v/(p^n)$.
        By~\cite[Theorem~3.3.10]{rideau2022model} (see also~\cite[Corollary 6.27]{rideau_thesis}), the theory $\cT$ eliminates $\VF$-quantifiers in the language $\Lold$.
        
        Note that all the predicates from the language $\LQE$ are already $\Lold$-definable, and hence for every $\LQE$-formula $\varphi(x)$ there exists an equivalent $\Lold$-formula $\psi(x)$ in the same free variables.
        By quantifier elimination in $\Lold$, we may assume that $\psi(x)$ does not contain any $\VF$-quantifiers.
        We now simply translate $\psi(x)$ back into the language $\LQE$ without introducing quantifiers over $\VF$.
        For this, we consider the various atomic formulas appearing in $\psi(x)$.
        Note that atomic formulas over $\VF$ and $\VG$ are already $\LQE$-formulas, so we only have to consider atomic formulas over the residue rings $R_n$.
        After some syntactic considerations, see \cite[Theorem~2.6.1]{Hod}, we may assume that all atomic formulas over some residue ring $R_n$ in $\psi(x)$ are of one of the following forms (or their negations):
        \begin{enumerate}
            \item $t(\zeta) = 0$ for some $\Lring$-term $t$ with variables $\zeta = (\zeta_1, \ldots, \zeta_m)$ ranging over the sort $R_n$,
            \item $s_n(t(\gamma)) = \zeta$ for some $\Loag$-term $t$ with variables $\gamma$ over $\Gamma^n$ and $\zeta$ over $R_n$,
            \item $\ac_n(t(x)) = \zeta$ for some $\Lring$-term $t$ with variables $x$ over $K^n$ and $\zeta$ over $R_n$,
            \item $\res_n(t(x)) = \zeta$ for some $\Lring$-term $t$ with variables $x$ over $K^n$ and $\zeta$ over $R_n$, and
            \item $\res_{n,m}(t(\xi)) = \zeta$ for some $\Lring$-term $t$ with variables $\xi$ over $R_n^k$ and $\zeta$ over $R_m$.
        \end{enumerate}
        For each variable $\zeta$ over $R_n$ appearing in one of the above terms we introduce new variables $\tilde{\zeta}_0, \ldots, \tilde{\zeta}_{n-1}$ over $R_1$.
        Every quantifier over $\zeta$ is replaced by the same quantifiers over the $\tilde{\zeta}_i$, and each of the atomic formulas above are replaced by the following formulas in the $\tilde{\zeta}_i$:
        \begin{enumerate}
            \item We replace the formula $t(\zeta_1, \ldots, \zeta_m) = 0$ for an $\Lring$-term $t$ by the $\LWitt$-formula $\ker_n(t^W((\tilde{\zeta})_{i,j});  0)$ over $R_1$.
            \item Replace $s_n(t(\gamma)) = \zeta$ by $S_n(t(\gamma), \tilde{\zeta}_1, \ldots, \tilde{\zeta}_n)$.
            \item Replace $\ac_n(t(x)) = \zeta$ by $\AC_n(t(x), \tilde{\zeta}_1, \ldots, \tilde{\zeta}_n)$.
            \item For a formula of the form $\res_n(t(x)) = \zeta$ for some $\Lring$-term $t$, note that if $v(t(x))\geq 0$ then $\res_n(t(x)) = \ac_n(t(x))s_n(v(t(x)))$. Hence this case follows from the previous three cases.
            \item Finally, we replace a formula $\res_{n,m}(t(\xi)) = \zeta$ by the $\LWitt$-formula 
            \[
            \ker_m(\res_{n,m}(t^W((\tilde{\xi})_{i,j})); (\tilde{\zeta}_1, \ldots, \tilde{\zeta}_m)).
            \]
        \end{enumerate}
        After these replacements, we obtain an $\LQE$-formula equivalent to $\varphi(x)$ which contains no quantifiers over $\VF$.
    \end{proof}

    \begin{rem}
        The above result is not uniform in $p$ as the construction of the Witt vectors is not uniform in $p$.
        It seems reasonable to expect that one can obtain a version of this result which is uniform in $p$ by using the big ring of Witt vectors, together with uniform relative quantifier elimination down to $\O_v/(\ell)$ over all integers $\ell$.
    \end{rem}

    It follows immediately from \cref{thm:qe} that $R_1\times \VG$ is stably embedded in $\LQE$.
    However, in contrast to \cref{kor:SE.ring}, $R_1$ is in general not stably embedded as an $\Lring$-structure in $(K,v)$ with the language $\LQE$:
    \begin{examp}
    Consider
    $K = \bC_p$.
    We take a compatible system $(x_q)_{q\in \Q}$ of roots of $p$, so that $x_q^{q^{-1}} = p$.
    On $R_1 = \O_{\mathbb{C}_p}/(p)$, we can then take $s_1\colon \Q_{\geq 0}\to R_1$ to be the map $q\mapsto \res_1(x_q)$ for $q < 1$ and $q\mapsto 0$ for $q \geq 1$.
    Let $X$ be the set $s_1([0,1)\cap \Q)$, which is definable in $\LQE$ as 
    \[
    X = \{\xi\in \O_v/(p)\mid \xi\neq 0 \text{ and } S_1(v(\xi), \xi) \text{ holds}\}.
    \]
    Then the inverse image
    \[
    \res_1^{-1}(X) = \bigsqcup_{q\in [0,1)\cap \Q} x_q + p\O_{\bC_p}
    \]
    is not a finite union of Swiss cheeses, and hence cannot be $\Lval(\bC_p)$-definable in $\bC_p$, by Holly's theorem~\cite[Corollary 3.8]{holly}.
    Thus also $X$ itself is not $\Lring$-definable, even with parameters.
    \end{examp}
    
    \section*{Acknowledgements}
 F.J.'s research was 
supported by the Deutsche Forschungsgemeinschaft (DFG, German Research Foundation)
via the ANR-DFG grant AKE Pact (project number 545528554) and a Visiting Research Fellow position at Merton College, Oxford. 
M.K. would like to thank the Max Planck Institute for Mathematics in Bonn for its hospitality and support.
F.V.'s research was supported by 
Humboldt Research Fellowship by the Alexander von Humboldt Stiftung.
All authors were supported under Germany's Excellence Strategy EXC 2044/2-390685587, Mathematics M\"unster: Dynamics--Geometry--Structure. Moreover, all 
authors were also partially funded under Germany Excellence Strategy-EXC-2047/1-390685813 via the Hausdorff Research Institute for Mathematics during their participation in the trimester programme Definability, Decidability, and Computability and thank HIM for its hospitality.

We wholeheartedly thank Sylvy Anscombe, Martin Bays, Philip Dittmann, Konstantinos Kartas, Silvain Rideau-Kikuchi, and Jonas van der Schaaf for sharing their insights in numerous helpful and
inspiring discussions around the topics of this paper.

    \section*{AI disclosure}
    Claude.ai (Opus 5) provided the technical details for \cref{rem:claude}, and 
    proofread an earlier version of this paper. Some of the 
    suggested changes were then implemented by hand by the authors. None of the
    proofs in the paper were done by Claude.ai.

\end{document}